\documentclass[12pt,a4paper]{amsart}
 \pdfoutput=1
\usepackage{amsmath}
 \usepackage{amsmath,amssymb,amsthm}
 \usepackage{bm}  
 \usepackage{mathtools}
 \usepackage{tikz}
 \usepackage{float}

\usepackage{enumitem}
 
 \usepackage[utf8]{inputenc}
  \usepackage{csquotes}
 \usepackage[T1]{fontenc}
 \usepackage{lmodern}
 \usepackage[babel]{microtype}
 \usepackage[english]{babel}

 \usepackage[giveninits=true, backend=biber, maxnames=99]{biblatex}
 \DeclareFieldFormat[article,incollection,inbook]{title}{#1}
 \usepackage{geometry}
\newcommand{\cm}{\mathcal{M}}

\newcommand{\cs}{\mathcal{S}}

\newcommand{\cv}{\mathcal{V}}
\newcommand{\V}{\mathcal{V}}

\newcommand{\ch}{\mathcal{H}}

\newcommand{\C}{\mathcal{C}}

\newcommand{\ci}{\mathcal{I}}

 \usepackage{enumitem}

\DeclareMathAlphabet{\mymathbb}{U}{BOONDOX-ds}{m}{n}

 \usepackage{xcolor} 	
 \usepackage{hyperref}
 \hypersetup{
 	colorlinks,
     linkcolor={red!60!black},
     citecolor={green!60!black},
     urlcolor={blue!60!black},
 }

\theoremstyle{plain}
\newtheorem{theorem}{Theorem}[section]

\newtheorem{claim}[theorem]{Claim}

\newtheorem{lemma}[theorem]{Lemma}

\newtheorem{observation}[theorem]{Observation}

\newtheorem{conjecture}[theorem]{Conjecture}
\newtheorem{problem}[theorem]{Problem}

\newtheorem*{theorem*}{Theorem}
\newtheorem*{corollary*}{Corollary}

\theoremstyle{definition}
\newtheorem{remark}[theorem]{Remark}

\title{Global and local degree conditions for matchability}

\author{Ron Aharoni}
\address{Ron Aharoni, Department of Mathematics, Technion, Haifa, Israel 32000}
\email{ra@tx.technion.ac.il}

\author{Eli Berger}
\address{Eli Berger Department of Mathematics, Haifa University, Haifa, Israel 31000}
\email{berger@math.haifa.ac.il}

\author{Attila Jo\'{o}}
\address{Attila Jo\'{o}, Department of Mathematics, Technion, Haifa, Israel 32000}
\email{a.joo@technion.ac.il}

 
\begin{document}
\begin{abstract}
A corollary of Hall's marriage theorem is that
  a sufficient condition for a list  $(V_1, \ldots ,V_m)$ of  sets  to  have a system of distinct representatives is 
 that  $|V_i|\ge deg_{\{V_1, \ldots ,V_m\}}(v)$ for every $i\in [m]$ and $v \in \bigcup_{i\in [m]}V_i$. This we dub a {\em global} condition. A 
 folklore result is that a {\em local} condition  -  that the inequality holds for pairs $i,v$ for which $v \in V_i$ - suffices. These are special cases of a 
 general
 type of results - large sets, whose elements are sparse in some sense, have a system of representatives that is independent in a related graph. We 
 study  two such scenarios, in both of which each $V_i$ is replaced by a  $k$-uniform hypergraph $H_i$, the representatives are hyperedges, and 
 distinctness is replaced by disjointness. In one setting the sparsity is measured by the degrees of vertices in the hypergraphs,  in the other by the 
 degrees of vertices in the line graph.  
 The proofs use the 
 topological version of Hall's theorem. In particular, we shall use a lower bound on the topological connectivity of the independence complex of a 
 graph, defined by vector representations.   We also provide short proofs of the local infinite version, 
 known as the ``Milner-Shelah theorem''.
 
\end{abstract}
\maketitle

\section{Introduction and summary}

        The protagonists  of this  paper are  choice functions, also known as ``transversals''. Let $\cv=(V_1, \ldots ,V_m)$ be an $m$-tuple of not 
        necessarily disjoint sets. A function $f: [m] \to \bigcup_{i\in [m]} V_i$ is a {\em choice  function} of $\cv$ if $f(i) \in V_i$ for all $i\in [m]$. 
        An injective 
        choice function is called an \textit{SDR} (``system of distinct representatives'').  If its image belongs to a given complex $\C$ we call it a 
        $\C$-{\em transversal}. For a graph $G$ let $\ci(G)$ be the complex of independent sets in $G$. An $\ci(G)$-transversal is called an {\em 
        independent transversal}, IT for short. For sets $S_1, \ldots,S_m$ we denote by  $\dot{\bigcup}_{j \le m} S_j$  the multiset that is the union of 
        the $S_i$s, with multiplicity. If the  $H_i$ are {\em hypergraphs}, i.e., collections of nonempty sets, then a transversal $f:[m] \to \bigcup_{i\in 
        [m]}H_i$ with $f(i)\cap f(j)=\emptyset$ whenever $i\neq j$ is called an SJR (system of disjoint representatives).

We identify any multihypergraph (edges possibly repeating) with its multiset of edges. For a multihypergraph $H$ and $v \in V(H)$  let 
$deg_H(v)=|\{e \in H\mid v \in e\}|$. Let $\Delta(H)=\max\{deg_H(v) \mid v\in V(H)\}$ and for a set $S$ of vertices, let $\Delta_H(S)= \max_{s 
\in S}deg_H(s)$. So, $\Delta(H)=\Delta_H(V(H))$.
The  {\em matching complex} (sets of disjoint edges) in $H$ is denoted by $\cm(H)$. Given a multihypergraph $H$, the {\em line-graph} $L(H)$ 
of $H$ is a graph whose vertices are the hyperedges of $H$ (with multiplicity) and two vertices are connected if the corresponding hyperedges meet.

The theme of the paper is ``large sets of sparse elements have a $\C$-transversal''. Here ``large'' is simply in the counting sense. Sparsity, on the 
other hand, can have one of a few meanings. 
\begin{enumerate}[label=(\Roman*)]
    \item \label{item: small DeltaV} Having small $\Delta(\V)$.
    \item\label{item: small DeltaG}  When $\C=\ci(G)$, having small $\Delta(G)$.
    \item\label{item: small DeltaHyper} When the $H_i$s are  hypergraphs, having small $\Delta(\dot{\bigcup}_{i\in [m]}H_i)$.
    \item\label{item: small DeltaL} When the $H_i$s are  hypergraphs, having small $\Delta(L(\dot{\bigcup}_{i\in [m]}H_i))$.
\end{enumerate}

Yet another distinction to be made: in the condition on size vs. sparsity there are {\em global} vs. {\em local} conditions. In the global case the size 
of all $V_i$ is compared with the sparsity of all elements. In the local case the size of each $V_i$ is compared with the sparsity of its own  
elements.  We prove global and local results  when the hypergraphs $V_i$ are $k$-uniform for fixed $k$. We also study such results for special 
classes of graphs.

\section{Tools}

A \emph{simplicial complex} $\C$ on vertex set $V$ is a collection of subsets of $V$ that is closed-down, i.e., if $\sigma\in \C$, then $\tau\in \C$ 
for every $\tau\subseteq \sigma$. The elements of $\C$ are called 
\emph{simplices}.  Let 
     $\C$  be a  complex on
 $V:=V_{[m]}$. A $\C$-transversal is a transversal, whose image belongs to $\C$. 
For a multihypergraph $H$, the independence complex $\mathcal{I}(H)$ of $H$ consists of those  $I\subseteq V(H)$ for which there is no $e\in 
H$ such that $e \subseteq I$. An  
 $\ci(H)$-transversal is called an {\em independent transversal}, \textit{IT} for short.  The {\em matching complex} $\cm(H)$ of $H$ has as 
 simplices the matchings of $H$ (i.e. sets of pairwise disjoint hyperedges). Note that $\cm(H)=\ci(L(H))$.

 Let $\mathcal{H}=(H_1,\dots, 
 H_m)$, 
 where  each $H_i $  is a $k$-uniform hypergraph and let $H:=\dot{\bigcup}_{i\in [m]}H_i$. Then an SJR (system of disjoint representatives) is an 
 $\cm(H)$-transversal. (Another frequently used term is ``rainbow matching'').
 For a vertex $v$ we denote by $deg_H(v)$ the number of edges of $H$ containing $v$. Let $\Delta(H)=\max_{v \in V(H)}deg_H(v)$.
The matching number and fractional matching number of a hypergraph $H$ are denoted by $\nu(H)$ and $\nu^{*}(H)$, respectively.

The main tool we shall use throughout the paper is the so-called ``topological Hall 
 theorem''. It uses a connectivity parameter denoted by $\eta$, which is the customary connectivity parameter  in topology, +2 (the addition of $2$ 
 simplifies the statement of results). For a simplicial complex 
$\C$, $\eta(\C)$ is $1$ plus the minimal integer $m$ such that the $m$-th homology group of $\C$ is non-trivial.   In homotopic (for us - 
combinatorial) terms, rather than homological, $\eta(\C)$ is the maximal $m$ for which any copy of a sphere of any dimension $d-1\le m-1$ in 
$\C$ can be filled with simplices of size at most $d$ - this may give some clue as to the combinatorial significance of the notion. 
For example, \begin{enumerate}
    \item 
$\eta(\C)\ge 2$ if and only if $\C$ is non-empty and connected, 
meaning that any copy $S$ of $S^0$ (two vertices) can be completed to a path with boundary $S$ by simplices, the size of each of which is $2$.
\item
$\eta(\C)=\infty$ whenever $\C$ is contractible.
\item $\eta(S^k)=k+1$. 
\end{enumerate}
Given a collection $\mathcal{V}=(V_1,\dots,V_m)$ of vertex sets and $J \subseteq [m]$, let  
        $V_J=\bigcup_{j \in J} V_j$.
\begin{theorem}[Topological  Hall]\label{tophall}
    Let $\C$ be a complex on the finite set $V=\bigcup_{i \in [m]}V_i$. If $\eta(\C[V_J]) \ge |J|$ for every $J \subseteq 
    [m]$ then there exists 
    a 
    $\C$-transversal.
\end{theorem}

This  was implicitly proved in \cite{aharoni2000hall}, and was explicitly formulated by the first author, as cited in
\cite{meshulam2001clique}.

\begin{remark}
    All that is written here applies both to homological and homotopic connectivity (that are anyway the same whenever $\eta\ge 3$, by a theorem of 
    Hurewicz).
\end{remark}
To apply the theorem, we need combinatorially-formulated lower bounds on $\eta$. When the complex is $\ci(G)$ for some graph $G$, most 
known lower bounds are  formulated in terms of {\em domination numbers}. The neighbourhood of a set $S$ of vertices in a graph $G$ is denoted 
by $N_G(S)$.  A set $S$ is said to {\em dominate} a set $T$ if $N_G(S)\cup S \supseteq T$. A set dominating $V(G)$ is plainly said to be {\em 
dominating}.
It is said to be {\em totally dominating} if $\bigcup_{s \in S}N_G(s)=V(G)$.

Here are six parameters that are useful for defining lower bounds on $\eta$:

\begin{enumerate}
    \item $\gamma(G)$ is the minimal size of a dominating set of vertices, namely a set $S$ for which $N_G(S)\cup S=V(G)$. 
    \item $\gamma_t(G)$ is the minimal size of a totally dominating set of vertices, namely a set $S$ for which $\bigcup_{s \in S}N_G(s)=V(G)$. 
    \item $\gamma^i(G)$ - the maximum, over all independent sets $I$, of the minimal size of a set $S$ dominating $I$.
    \item $\gamma_i^i(G)$ - the maximum, over all independent sets $I$, of the minimal size of an {\em independent} set $S$ dominating $I$. 
    Obviously, $\gamma_i^i(G) \ge \gamma^i(G)$.

        \item $\Gamma(G)$ - the vector representation domination number of $G$. Being our main tool, we devote to it below a special subsection. 
         \item $\lambda (L(G))$, the largest eigenvalue of the Laplacian of the graph.

        \end{enumerate}
        The relationship with $\eta$ is given by:

\begin{theorem}\label{inequalities}\hfill
    \begin{enumerate}[label=(\alph*)]
        \item\label{item: gammat} $\eta(\ci(G))\ge \frac{\gamma_t(G)}{2}$.

        \item\label{item: gammaii} $\eta(\ci(G))\ge \gamma_i^i(G)$ (implicit  in \cite{aharoni2000hall}).

        \item\label{item: GammaG} $\eta(\ci(G))\ge \Gamma(G)$. (\cite{aharoni2005eigenvalues})

        \item\label{item: frac} $\eta(\ci(G))\ge \frac{|V(G)|}{\lambda (L(G))}$ (\cite{aharoni2005eigenvalues}).
    \end{enumerate}
\end{theorem}
        \subsection{Vector representation of graphs.}
        A {\em vector representation} of $G$ is an assignment of a vector $P(v)$ in some fixed $\mathbb{R}^n$ to each  vertex $v$, satisfying the 
        condition that
        $P(x)P(y) \ge 1$ whenever $xy\in E(G)$ and  $P(x)P(y) \ge 0$ otherwise. A vector representation can be considered as a matrix $P$, whose 
        rows are the vectors $P(v)$ in some ordering.
        
        If $G$ is  the line graph of a hypergraph  $H$  then the canonical representation in $\mathbb{R}^{V(H)}$ maps each vertex (i.e. hyperedge) 
        to its characteristic vector.  Any graph $G$ is the line graph of the hypergraph of its stars (viewed as sets of edges). By denoting 
        $\mathsf{star}_G(v)$ the set of edges incident with $v$, a canonical representation of a general graph $G$ can be defined by 
        $P(v)=\chi_{\mathsf{star}_G(v)} \in \mathbb{R}^{E(G)}$ for $v\in V(G)$. 
        
        A vector $\vec{\alpha}\in \mathbb{R}_{+}^{V}$ is said to  {\em dominate}  a representation $P$ if $(\sum_{v\in V}\alpha(v)P(v))\cdot 
        P(u)\ge 1$ for every $u \in V$. We define $\Gamma(G)$ to be the supremum, over all representations $P$ of $G$, of $\min\{\vec{\alpha}\cdot 
        \vec{1} \mid ~\vec{\alpha}  \text {~dominates ~} P\}$. Let $Adj(G)$ be the adjacency matrix of $G$. 
        An equivalent definition is 
        \[\Gamma(G)=\sup_{B~ p.s.d, B\ge Adj(G)}
        \min\{\vec{1}\cdot \vec{\alpha} \mid \vec{\alpha} B \ge \vec{1}, \vec{\alpha} \geq \vec{0}\}, \]
        where p.s.d. stands for positive semidefinite.
        By linear programming duality we also have 
        \begin{equation}\label{Gamma}
           \Gamma(G)=\sup_{B~ p.s.d, B\ge Adj(G)}\max
           \{\vec{1}\cdot \vec{\alpha} \mid \vec{\alpha} B \le \vec{1}, \vec{\alpha} \geq \vec{0}\}. 
        \end{equation} \hfill

       (see \cite{zewi2012vector}).

Summarizing \ref{item: GammaG}:

    \begin{equation}
           \begin{gathered}
    \eta(\ci(G))\ge \sup_{P \text{~is a representation of }G}\max\{\vec{\alpha}\cdot \vec{1} \mid \vec{\alpha}PP^T \le \vec{1},\ \vec{\alpha} \geq 
    \vec{0} \}=\\ \newline \sup_{P \text{~is a representation of }G}\min\{\vec{\alpha}\cdot \vec{1} \mid \vec{\alpha}PP^T \ge \vec{1},\ 
    \vec{\alpha} \geq \vec{0} \}. 
    \end{gathered}\end{equation}
    

In \cite{aharoni2000hall} the following was implicitly proved:
\begin{theorem}\label{gammailine}
    If $H$ is a $k$-uniform multihypergraph then $\eta(\mathcal{M}(H))\ge \nu(H)/k$. 
\end{theorem}
 Applying \ref{item: GammaG} to the canonical representation of $L(H)$  yields the stronger: 

\begin{theorem}[\cite{aharoni2005eigenvalues}]\label{gammailine extra}
    If $H$ is a $k$-uniform multihypergraph then $\eta(\mathcal{M}(H))\ge \nu^*(H)/k$.\end{theorem}
\begin{proof}
Let $G=L(H)$ and let $P$ be the canonical representation of $G$ written as a matrix.  
Let $\vec{\alpha}$ be a non-negative row vector on the hyperedges of $H$ satisfying $\vec{\alpha}\cdot PP^T \ge \vec{1}$
that minimizes $\vec{\alpha}\cdot \vec{1}$. Denote $t$ this minimum.
By the definition of $\Gamma$ as a supremum over all representations, and by Theorem \ref{inequalities} \ref{item: GammaG}, it suffices to prove 
that
$ t \ge \nu^*(H)/k$.
 By the definition of $P$, the vector $\vec{q}=\vec{\alpha} P$ is a fractional cover of $H$, so $\nu^*(H)=\tau^*(H) \le \vec{q}\cdot \vec{1}$. 
 Since each row of $P$ contains exactly $k$ many $1$ entries, 
$\vec{q}\cdot \vec{1} \le  k\vec{\alpha} \cdot \vec{1}$. Combining these inequalities yields 
$t \ge \nu^*(H)/k$, proving  the theorem. 
\end{proof}


Another tool we shall use is a theorem of Caro and (independently) Wei. Below $\alpha(G)$ is the maximal size of an independent set in a graph 
$G$. 
    
\begin{theorem}[\cite{wei1981lower, caro1979new}]\label{carowei}
 In any  graph $G$, 
\[ \alpha(G) \geq  \sum_{v\in V} \frac{1}{deg_G(v)+1}. \]
\end{theorem}

For intuition, note that in $d$-regular graphs this is obvious, in fact $\chi(G)\le d+1$, meaning that $V(G)$ can be covered by independent sets 
whose average size is $\frac{|V|}{d+1}$.

        

\section{SJRs - global and local conditions for their existence}
As already mentioned, Hall's theorem  implies
results of type \ref{item: small DeltaV}, in which the condition is on the size of the $V_i$s, compared with their popularity. 
 \begin{theorem}\label{thm: global weakest}
    If the sets $V_i$ are nonempty and $|V_i|\ge deg_{\cv}(v)$ for every $i\in [m]$ and $v \in \bigcup_{i\in [m]}V_i$ then $(V_1,V_2, 
    \ldots,V_m)$ has an SDR. 
\end{theorem}
In bipartite graph terminology:  if in a bipartite graph with sides $M,W$ all degrees of vertices in $M$ are at least as large as all degrees in $W$, 
then there is a matching covering $M$. The translation to the set terminology is done by identifying each $m \in M$ with the set of its neighbours.

 A folklore result is that the local condition also suffices.

\begin{theorem}\label{localms}
    If the sets $V_i$ are nonempty and $|V_i|\ge deg_{\cv}(v)$ for every $i\in [m]$ and $v \in V_i$, then there exists an SDR. 
\end{theorem}
In graph terminology, this says that it suffices to assume $deg(m) \ge deg(w)$ whenever $mw$ is an edge with $m\in M$ and $w\in W$. 
Later on we shall meet three proofs of this fact, including a topological one (special case of Theorem \ref{coarselocal}).

Our first aim is to extend these results to  the case where the vertex sets $V_i$ are replaced by $k$-uniform hypergraphs $H_i$. The condition of 
distinctness is replaced by that of disjointness, so the aim is to guarantee the existence of SJRs. The following, global result, generalizes Theorem 
\ref{thm: global weakest}, which is its  $k=1$ case.

\begin{theorem}\label{coarseglobal}
    Let $\mathcal{H}=(H_1, \ldots, H_m)$, where each $H_i$ is a $k$-uniform hypergraph on the common vertex set $V$, and let 
    $H=\dot{\bigcup}_{j \in [m]} H_j$ (counting with multiplicity). If $|H_i| \ge k\Delta(H)$ for every $i \in [m]$ then there exists an SJR. 
\end{theorem}
\begin{proof}
    For $J \subseteq [m]$ we have $|\mathcal H_J|\ge k|J|\Delta(\mathcal H_J)$, where $\mathcal H_J$ denotes $\dot{\bigcup}_{j \in J} H_j$, i.e.,  
    the $k$-uniform multihypergraph that is the disjoint union of $H_j$ for $j\in J$. The fractional matching obtained by assigning weight 
    $\frac{1}{\Delta(\mathcal H_J)}$ on every edge of $\mathcal H_J$, shows that $\nu^*(\mathcal H_J) \ge k|J|$. Hence by Theorem 
    \ref{gammailine extra} 
    $\eta(\mathcal{M}(\mathcal H_J))\ge |J|$, yielding the sufficient condition  in Theorem \ref{tophall} for the existence of an SJR.
\end{proof}
    Another proof is via the stronger local result that generalizes Theorem \ref{localms}. It is stated and proved next.
    
\begin{theorem}\label{coarselocal}
  Let $\mathcal{H}=(H_1, \ldots, H_m)$, where each $H_i$ is a $k$-uniform hypergraph on the common vertex set $V$ and let 
  $H=\dot{\bigcup}_{j \in [m]} H_j$. If $|H_i|\ge k\Delta_H(\bigcup H_i)$ for every $i\in [m]$ then there exists an SJR.
\end{theorem}

\begin{proof}  
 Let $J \subseteq [m]$ and $\ch_J:=\dot{\bigcup}_{j \in J}H_j$. Clearly,  $ \ch_J$ is a subhypergraph of $H$. Let $P$ be the canonical 
 representation of $L(\ch_J)$, i.e., $P(e)=\chi_e$, where $\chi_e$ denotes the characteristic function of $e\in \ch_J$. Note that 
 $\mathcal{I}(L(\ch_J))=\mathcal{M}(\ch_J)$.  Consider $P$ as a $ \ch_J \times V $  matrix whose $e$-th row is $P(e)$.
  Define a row vector $\vec{\alpha}\in \mathbb{R}^{\ch_J}$    by $\vec{\alpha}(e) = \frac{1}{k \Delta_H(\bigcup H_i)}$,  whenever $e \in H_i$. 
By the assumption of the theorem
$$\vec{\alpha} P \leq \frac{1}{k} \vec{1}_V$$ and since every row of $P$ has $k$ many $1$s, we have $\vec{\alpha} PP^T \leq \vec{1}$. 

By applying part \ref{item: GammaG} of Theorem \ref{inequalities} to $\mathcal{I}(L(\ch_J))$,  we conclude $$ \eta(\mathcal{M}(\ch_J)) \geq  
\vec{\alpha} \cdot \vec{1} = \sum_{i\in J} \frac{|H_i|}{k \Delta_H(\bigcup H_i)}.$$
By assumption,  $|H_i| \geq k \Delta_H(\bigcup H_i)$, hence  $\eta(\mathcal{M}(\mathcal H_J)) \geq |J|$. The theorem now follows by Theorem 
\ref{tophall}. 
\end{proof}

 \subsection{Sharpness}
Theorem \ref{coarseglobal}  is asymptotically sharp. The factor $k$ in the condition cannot be lowered to $k \left(1-\frac{1}{k^2} \right)$. The 
examples showing this are hypergraph versions of a construction due to Yuster 
\cite{yuster1997independent} 
and (independently) Jin \cite{jin1992complete},  generalized in \cite{szabo2006extremal}. These constructions can be summarized as follows:

\begin{theorem}\label{jyst}
Let 
 $G$ be the disjoint union of  
 \(2d-1\) copies $B_1, \ldots ,B_{2d-1}$ of \(K_{d,d}\).  There is a partition of $V(G)$ into $2d$ sets
$V_1, \ldots ,V_{2d}$ of size $2d-1$ compatible with the bipartitions of the graphs $B_i$ in the sense that if two vertices are in distinct sides of 
$B_i$ then they 
lie in distinct $V_i$s, such that there does not exist an IT.  
 \end{theorem}
 We shall call such a graph, together with the partition, a $d$-{\em JYTS contraption}. The idea of the following construction is to start with a 
 $d$-{\em JYTS contraption} for a certain $d$ and replace each vertex with a $k$-set in such a way that the $k$-sets corresponding to the endpoints 
 of any edge meet while the maximal degree is relatively low. 
 
 We use the standard notation $AG(2,q)$ for the affine plane of order $q$, obtained from the $q+1$-uniform projective plane by removing 
 the vertices of one line.  Let $k$ be a power of $2$. It is known that $A(2,k)$ exists. Fix a line $\ell_0$ of $A(2,k)$ and let
 \[ H:=\{ \ell:\ \ell \text{ is a line of }A(2,k), \  \left|\ell \cap \ell_0 \right|=1 \}. \]
 We claim that $H$ is a $k$-uniform $k$-regular hypergraph with $L(H)= K_{k,k,\ldots,k}$ ($k$ times). Indeed, the $k^2+k$ lines of $A(2,k)$ fall 
 into $k+1$ equivalence classes of $k$ lines, where the equivalence is parallelism, and $H$ was defined by removing one such class.
 Let $d=k^2/2.$ Obviously, $L(H)$ contains  $K_{d,d}$ as a subgraph. The $d$-JYST contraption consists of $2d-1=k^2-1$ copies of $K_{d,d}$. 
 Replace each copy of $K_{d,d}$ with a copy $H^j$ of $H$ where every vertex $v$ of $K_{d,d}$ is replaced by a hyperedge $\Phi_j(v)\in H^j$ 
 such that $\Phi_j(v)\cap \Phi_j(w)\neq \emptyset $ whenever $vw$ is an edge of $K_{d,d}$. Let $H^+$ be the union of the vertex-disjoint 
 hypergraphs $H^j$ and let $H_i\ (i\in[2d])$ be the partition that $H^+$ inherits  from the $d$-JYST contraption. On the one hand, there is no SJR 
 because that would provide an IT in the $d$-JYTS contraption.
On the other hand, $\Delta(H^+)=\Delta(H)=k$ and $|H_i|=2d-1$ for each $i\in [2d]$. This provides the ratio: \[\frac{2d-1}{k}=\frac{k^2-1}{k}=k 
\left(1-\frac{1}{k^2} \right).  \]

\section{Independent transversals in graphs}

Next we turn to type \ref{item: small DeltaG}. The following is a  classical result of Haxell. It was originally proved combinatorially, but it also 
follows  from part \ref{item: gammat} of Theorem \ref{inequalities}:
\begin{theorem}[\cite{haxell1995condition}]\label{thm:pennyfirst}  Let $V_1, \ldots, V_m$ be a partition of the vertex set $V$ of a graph $G$. 
  If  $|V_i| \ge 2\Delta(G)$  for every $i \in [m]$  then there exists an IT.
\end{theorem}

\begin{remark}\label{rem: test}
    This is also true if the sets $V_i$ are not disjoint, once the degrees are counted with multiplicity - an edge $xy$ is counted towards $deg(x)$  the 
number of sets $V_i$ containing $y$. Namely, defining
\[ deg_t(x):=\sum_{xy \in E}deg_\cv(y)      \]
and $\Delta_t(G):=\max_{x \in V} deg_t(x)$, if $|V_i|\ge 2\Delta_t(G)$ for all $i$ then there exists an IT.
\end{remark}

 Theorem \ref{jyst} states that this theorem is sharp, namely 
for every $d$ there exists a graph $G$ with $\Delta(G)=d$ and a partition  of $V(G)$ into parts of size  $2d-1$ with no IT.

In this section, we discuss improvements of Theorem \ref{thm:pennyfirst} for certain classes of graphs. The first result in this direction appeared in 
\cite{aharoni2007independent}:  
\begin{theorem}[\cite{aharoni2002tree}]
If $G$ is chordal (having no induced cycles larger than $3$) and $|V_i|>\Delta(G)$ then there exists an IT.
\end{theorem}
In fact, we do not know a  counterexample to the daring idea that it suffices to exclude  chordless $C_4$'s (i.e., allowing chord-less cycles of length 
>4).

In \cite{aharoni2015cooperative} the following was proved:
\begin{theorem}\label{kdd}
    A graph witnessing the sharpness of Theorem \ref{thm:pennyfirst} with maximal degree $d$ must contain $2d-1$  connected components 
    isomorphic to $K_{d,d}$.  
\end{theorem}

This indicates that in graphs not containing induced copies of complete bipartite graphs of small orders Theorem \ref{thm:pennyfirst} may be 
strengthened. A particular case is line graphs. The line graph of a $k$-uniform hypergraph does not contain an induced $K_{1,k+1}$. This brings us 
to the next subsection. 
\subsection{Matching complexes (independence complexes of line graphs)}\hfill
\\

In \cite{aharoni2016eigenvalues} the following was shown:
\begin{theorem}
If a graph $G$ does not contain an induced $K_{1,k}$ and  $|V_i|\ge \Delta(G)+k$ then there exists an IT. 
    \end{theorem}

   This follows from  Theorem \ref{tophall} and the following:

    \begin{theorem}[\cite{aharoni2016eigenvalues}]\label{lineofgraphs}
     If $L$ is a line graph of a simple graph then $\eta(\ci(L))\ge \frac{|V(L)|}{\Delta(L)+2}$.
\end{theorem}

A generalization to multigraphs was obtained by the second author using topology.
For a multigraph $G$ let 
$\Delta_2(G)=\max_{e=xy \in E(G)}(deg(x)+deg(y))$.

\begin{theorem}[\cite{Eli2026private}]
For every finite multigraph $G$,    $\eta(\mathcal{M}(G))\ge \frac{|E(G)|}{\Delta_2(G)}$.
\end{theorem}

A possible strengthening:
\begin{conjecture}
   In any multigraph $G$ apart from $C_5$, $\gamma^i_i(L(G)) \ge 
   \frac{|E(G)|}{\Delta_2(G)}$.
\end{conjecture}

A slight adaptation of the proof of the last result, as it appears in \cite{aharoni2016eigenvalues}, yields a more general result. Recall that  a 
hypergraph is {\em linear} if no two edges in it share more than one vertex.   Theorem \ref{lineofgraphs} is a special case of: 
\begin{theorem}\label{lineoflinear}
     If $L$ is a line graph of a $k$-uniform linear hypergraph  $H$ then $\eta(\ci(L))\ge \frac{|V(L)|}{\Delta(L)+k}$.
\end{theorem}

 This follows from a careful inspection of  the proof of Theorem \ref{lineofgraphs} in \cite{aharoni2016eigenvalues}. That proof used part 
 \ref{item: frac} of Theorem \ref{inequalities} (so, it is algebraic).  Here is a combinatorial, $\Gamma$-based proof.
\begin{proof}
 
 Let $P$ be the matrix of the natural vector representation of $L$, i.e., the $e$-th row is $P(e)=\chi_e$ for every $e \in H=V(L)$. By part \ref{item: 
 GammaG} of Theorem \ref{inequalities}, in the version viewing $\Gamma$ as a minimum linear program, it suffices to show that if a non-negative 
 row vector $\vec{\alpha}$ on $V(L)=H$ satisfies

\begin{equation}\label{alpha}
    \vec{\alpha} PP^T \ge \vec{1}_H.
    \end{equation}
then $\vec{\alpha} \cdot \vec{1}_H \geq \frac{|V(L)|}{\Delta(L)+k}$. 
     Summing the inequalities of \eqref{alpha},  on the right hand side we have $|V(L)|$. On the left hand side 
      each term $\alpha(e)$ is multiplied  by $\sum_{f \in H}|f \cap e|$. Since $|e \cap e|=k$ and for each of the at most $\Delta(L)$ neighbours of $e$ 
      in $L$ we have $|f\cap e|=1$, this yields the desired conclusion. 
\end{proof}

\begin{conjecture}
    If $G$ has no minor isomorphic to $K_{3,3}$ (in particular if $G$ is planar) and $|V_i| > \Delta(G)+1$ then there is an IT.
\end{conjecture}

 \section{Type (IV): sparsity measured in the line graph}
Theorem \ref{thm:pennyfirst}  yields the following global result:

\begin{theorem}\label{refined_global_result}
Let $\mathcal{H}=(H_1,\dots, H_m)$, where each $H_i$ is a $k$-uniform hypergraph, and let $H=\dot{\bigcup}_{i \in [m]}H_i$. If   $|H_i| \ge 
2\Delta(L(H))$  for every $i \in [m]$  then there exists an SJR.
\end{theorem}

As expected, the local case needs a stronger condition.

 \begin{theorem}\label{localrefined}
         Let $\mathcal{H}=(H_1,\dots, H_m)$, where each $H_i$ is a $k$-uniform hypergraph.  Let $H=\dot{\bigcup}_{i \in [m]}H_i$ and let 
         $L=L(H)$. If  $|H_i|\ge k (deg_L(e)+1)$ whenever  $e\in H_i$ then there exists an SJR.
\end{theorem}
Note that $deg_L(e)+1$ may be asymptotically as large as $k\Delta(H)$, so there is a potential $k$-factor between the coefficients in the 
inequalities in Theorems \ref{coarselocal} and \ref{localrefined}. 

\begin{proof} 
By Theorem \ref{tophall}, it suffices to show that 
$\eta(\cm(\mathcal H_J))\ge |J|$ for every  $J \subseteq [m]$, where $\mathcal H_J$ denotes $\dot{\bigcup}_{j\in J}H_j$ . Write the inequalities in 
the condition of the theorem as
\[ \frac{1}{deg_L(e)+1}\ge \frac{k}{|H_j|}, \]
and sum over all edges $e\in \mathcal H_J$ (with multiplicity). On the right-hand side we get $k|J|$, and by Theorem \ref{carowei} and 
$deg_{L(\mathcal H_J)}(e) \leq deg_L(e)$, the sum on the 
left-hand 
side is at most $\alpha(L(\mathcal H_J))=\nu(\mathcal H_J)$. So, 
\[\nu(\mathcal H_J) \ge k|J|.\]

Applying Theorem \ref{gammailine} yields $\eta(\cm(\mathcal H_J))\ge |J|$. 
 \end{proof}
        How sharp is this result? For $m=2$ the factor $k$ is not needed - it suffices to assume that $|H_j|\ge deg_L(e)+1$ whenever  $e\in H_j$. 
        Indeed, fix $e \in H_1$. If $|H_2|>deg_L(e)$, then one of the elements of $H_2$ is not connected in $L$ to $e$, so there is an SJR. So, we 
        may assume $|H_2|\le deg_L(e) <|H_1|$. Similarly, we have $|H_1|<|H_2|$, a contradiction. 
        
         For $m>2$ 
        this condition is not sufficient, for example in the d-JYST contraption we have $\Delta(L)=d, |H_i|=2d-1$ and there is no 
        SJR. This shows that 
 a condition of the form $|H_i|>C deg_L(e) $        for $e\in H_i$ would not suffice, if $C<2$.

Theorem  \ref{thm:pennyfirst} has no local version:
\begin{observation}\label{stars}
For every $K\in \mathbb{N}$ there exists a graph $G$ and a partition $\cv$ of $V(G)$ such that $\frac{|V_i|}{deg(x)} \ge K$ for every $x \in V_i$ 
and yet there is 
no IT.    
\end{observation}
\begin{proof}
    Let $V_0= [K^2]\times \{ 0 \} $,  $V_i=[K]\times \{ i \}$ for $1\le i \le K^2$, and a vertex $(i,0) \in V_0$ is connected exactly to the vertices in 
    $V_i$.
\end{proof}

But Theorem \ref{lineofgraphs} may have a local version. We do not know an example in which $G$ is a line graph of a graph, and the condition 
$|V_i|>\Delta_G(V_i)+1$ does not suffice to guarantee the existence of an IT.

\section{Short proofs of the countable Milner-Shelah theorem}

Next we turn to the infinite version of the problem.  Theorem \ref{localms} remains true if both the set sizes and the number of sets are infinite. 
This was proved  in the countable case independently by 
Bollobás and Milner  \cite{bollobas1973theorem}, and by Shelah  
\cite{shelah1974substitute}.  The general  
case was proved by Milner and Shelah \cite[Theorem 3]{milner1974sufficiency}. A simpler proof  was given by Tverberg (see 
\cite{tverberg1976milner}). Yet another  proof was provided by Podewski and 
Steffens (see \cite[p. 166]{podewski1976injective}). Our contribution is a simple elementary proof,  and an even simpler proof based on the 
``marriage theorem'', i.e., a characterization of the existence of transversals \cite[Theorem 7]{podewski1976injective}. 

For this purpose, we switch to graph-theoretic, rather than transversal-theoretic terminology, one reason being that this is the terminology used in 
the existing literature on the subject.  Following the terminology in 
\cite{aharoni1983general},  $G=(M,W,E)$  denotes a (possibly infinite) 
bipartite graph where the vertex 
classes are 
referred to as the set of \emph{men} and set of \emph{women}, respectively.
An \emph{espousal} of $G$ is a matching (i.e., a set $F\subseteq E$ of pairwise disjoint edges) that covers $M$. A bipartite graph $G$ is 
\emph{espousable} if 
it has an espousal and \emph{inespousable} otherwise.   From a family 
$\{ V_i:\ i\in I \}$ of sets, one can construct a bipartite graph $G=(M,W,E)$ by letting $M:=\{ V_i:\ i\in I \}$ and $W:=\bigcup_{i\in I}V_i$ as 
well as $\{ V_i, v 
\}\in E$ if and only if $v\in V_i$. Then an SDR for $\{ V_i:\ i\in I \}$ corresponds to an espousal of $G=(M,W,E)$  and vice versa.  The condition 
in 
Theorem \ref{localms}, which we will refer to as  the \emph{Milner-Shelah condition}, can be formulated as follows:  `there is no isolated vertex 
in $M$ and for every 
edge $\{ m,w \}\in E$ with $m\in M$ and $w\in W$, we have $deg_G(m) \geq deg_G(w)$', where $deg_G(v)$ denotes the degree of  $v$ in $G$. 

\begin{theorem}[{Milner and Shelah, \cite[Theorem 3]{milner1974sufficiency}}]\label{thm: MilnerShelah}
If $G=(M,W,E)$  is a (possibly infinite) bipartite graph that satisfies the Milner-Shelah condition, then  $G$ is espousable.
\end{theorem}

The classical proof for finite graphs (originated in \cite{graham1969some}) reads as follows. Given a subset $X$ of $M$, we have 
\[ \left|X \right|=\sum_{\{ m,w \} \in E, m\in X}\frac{1}{deg_G(m)}\leq  \sum_{\{ m,w \} \in E, m\in X }\frac{1}{deg_G(w)}\leq
\sum_{\{ m,w \} \in E, w\in N_G(X) }\frac{1}{deg_G(w)}= \left|N_G(X) \right|, \]
where $N_G(X)$ denotes the set of women having a neighbour in $X$, and the inequalities follow, respectively, from the Milner-Shelah condition 
and from the fact 
that \[  \{ \{ m,w \} \in E, m\in X \}\subseteq \{ \{ m,w \} \in E, w\in N_G(X) \}. \] So, Hall's condition holds. An alternative proof, under slightly 
weaker assumptions, reads as follows:
\begin{theorem}\label{thm: MS finite}
Let $G=(M,W,E)$  be a finite bipartite graph and let $F\subseteq E$ be a maximum matching in $G$. If there is no isolated vertex in $M$ and 
$deg_G(m)\geq deg_G(w)$ holds for every $\{ m,w \}\in F$ with $m\in M$ and $w\in W$, then $M \subseteq \bigcup F$.
\end{theorem}
\begin{proof}
By Kőnig's theorem there exists a vertex cover $C$ with $\left|C \right|=\left|F \right|$. Then $C$ consists  of a single vertex from each edge 
of $
F$. For $X:=M\setminus C$ we  have $N_G(X)\subseteq C\cap W$ because $C$ is a vertex cover. In particular, every $w\in N_G(X)$ is an 
endpoint of some $e\in F$ 
whose other endpoint is not in $C$ and therefore lies in $X$.

  From the definition of 
  $N_G(X)$, it follows directly that \[ \sum_{m\in X}deg_G(m) \leq \sum_{w\in N_G(X)}deg_G(w). \] Summing the  inequalities in the 
  Milner-Shelah 
  condition 
  corresponding to
  the 
  edges of $F$ meeting $N_G(X)$ yields
  \[\sum_{w\in N_G(X)}deg_G(w)\leq \sum_{m\in X\cap \bigcup F}deg_G(m) \leq \sum_{m\in X}deg_G(m).  \]
  Combining these inequalities yields equality throughout. In particular,  $\sum_{m\in X\cap \bigcup F}deg_G(m) =
  \sum_{m\in X}deg_G(m)$. Since $deg_G(m)>0$ for every $m\in M$ by assumption, we conclude that  $X\cap \bigcup F=X$, i.e. 
  $X\subseteq \bigcup F$. We have seen that $C\subseteq \bigcup F$, in particular $M\setminus X=C\cap M\subseteq \bigcup F$. By 
  combining 
  these we get $M \subseteq 
  \bigcup F$.
  \end{proof}

 
From now on, we focus on the countable case. 
\subsection{An elementary proof} Suppose  that $G=(M,W,E)$ is a countable bipartite graph that satisfies the Milner-Shelah condition. Let 
us call a subgraph $G'=(M',W',E')$ \emph{saturated} if $N_{G'}(X)=N_G(X)$ for every $X\subseteq M'$. Note that saturated subgraphs inherit the 
Milner-Shelah condition.
\begin{claim}\label{clm: GsatisHall}
 $G$ satisfies Hall's condition. 
\end{claim}
\begin{proof}
 Suppose, for a contradiction, that there is an $X\subseteq M$ with $\left|N_G(X) \right|<\left|X \right|$. We may assume that $X$ is finite, since 
 otherwise we replace $X$ by an $(\left|N_G(X) \right|+1)$-element subset of $X$. Then 
 the saturated subgraph $G[X\cup N_G(X)]$ is  finite, satisfies the 
 Milner-Shelah condition but is inespousable. This contradicts the 
 already established finite case of the theorem. 
\end{proof}

\begin{lemma}\label{lem: key}
If $H$ is obtained  by deleting finitely many vertices from $G$ and $H$ satisfies Hall's condition, then for every $m\in M \cap V(H)$ there 
exists  a $w\in N_H(m)$ such that deleting $m$ and $w$ from $H$ yields a graph satisfying Hall’s condition.
\end{lemma}
\noindent Since $G$ satisfies Hall's condition (see Claim \ref{clm: GsatisHall}), the countable case now follows from Lemma \ref{lem: key} by a 
straightforward recursive 
construction: enumerate the vertices of $M$ and match 
them one by one, deleting matched vertices at each step.
\begin{proof}[Proof of Lemma \ref{lem: key}]
 Fix $m_0\in M \cap V(H)$ and suppose, for a contradiction, that no $w\in N_H(m_0)$ has the required property. Then for every $w\in N_H(m_0)$ 
 we choose a finite set
 $X_w\subseteq (M \cap V(H))\setminus \{ m_0 \} $ such that $w\in N_H(X_w)$ and $\left|X_w \right|=\left|N_H(X_w) \right|$.
 
 \begin{claim}\label{clm: critical union}
  For every finite $Y\subseteq N_H(m_0)$, for the set $X_{Y}:=\bigcup_{w\in Y}X_w$, $H[X_Y \cup N_H(X_{Y}) ]$ has a perfect 
  matching. 
  In particular $\left|X_{Y} \right|=\left|N_H(X_{Y}) \right|$. 
 \end{claim}
 \begin{proof}
  The finite saturated subgraph $H':=H[X_Y\cup N_H(X_{Y})]$ of $H$ satisfies Hall's condition, since $H$ does. Thus by Hall's theorem it 
  has an espousal 
  $F$. Then for every 
  $w\in Y$, the matching  $F$ covers in particular the finite set $X_w $. Hence, by $\left|X_w \right|=\left|N_H(X_w) \right|$, the matching $F$ 
  covers the set 
  $N_H(X_w)$ as 
  well. It follows that $F$ covers  $N_H(X_{Y})=\bigcup_{w\in Y}N_H(X_w)$ and hence $F$ is a perfect matching of $H'$.
 \end{proof}
If $deg_H(m_0)$ is finite, then by applying Claim \ref{clm: critical union} with $Y=N_H(m_0)$ we conclude that for $X:=X_{Y}\cup 
\{ m_0 \}$ we have $\left|X \right|=\left|N_H(X) \right|+1$.  This 
contradicts 
the 
assumption 
that 
$H$ satisfies Hall's condition. Therefore $deg_H(m_0)$ is infinite. 

 Let 

\[ n:=\sum_{w\in W\setminus V(H),\ deg_G(w)<\aleph_{0}} deg_G(w).  \]
Take an $(n+1)$-element subset $Y$ of $N_H(m_0)$ and let $X:=X_Y$. 
We claim that no $m\in X$ is adjacent in $G$ to a vertex $w$ of infinite degree. Indeed, if such a vertex $w$ existed, then the Milner-Shelah 
condition would imply that
$deg_G(m)=\aleph_{0}$ and hence $deg_H(m)=\aleph_{0}$, contradicting the fact that $X$ is a finite set with $\left|X \right|=\left|N_H(X) 
\right|$. 
Therefore
\[ N_G(X)\setminus N_H(X)\subseteq \{w\in W\setminus V(H):\ deg_G(w)<\aleph_{0} \} \]
and thus \[ (n+1)+\sum_{m\in X}deg_H(m)>\sum_{m\in X}deg_G(m). \]
Since   $m_0\notin X$  and $m_0$ has at least $n+1$ neighbours in $N_H(X)$ in $H$  by  $  N_H(X) \supseteq Y $, we get:
\[ \sum_{w\in N_H(X)} deg_G(w) \geq \sum_{w\in N_H(X)} deg_H(w)\geq  (n+1)+ \sum_{m\in X}deg_H(m)> \sum_{m\in X}deg_G(m).\]

On the other hand, the finite saturated subgraph $H[X\cup N_H(X)]$ of $H$  admits a perfect matching $F$ by Claim \ref{clm: critical union} and 
by summing the inequalities in the 
Milner-Shelah 
condition for the edges of $F$ in the graph $G$  we conclude
\[\sum_{w\in N_H(X)} deg_G(w) \leq \sum_{m\in X}deg_G(m), \]
a contradiction.
\end{proof}
\subsection{A proof based on the  Podewski-Steffens theorem}
We call $G=(M,W,E)$ \emph{critical} if 
it is espousable and  every espousal covers all vertices of $W$. A graph  $G=(M,W,E)$  is a 
$1$\emph{-obstruction}  if there is an $m\in M$ such that $G-m$ is critical. A matching $F$ \emph{fits} a $1$\emph{-obstruction} $G$ if there 
is an $m\in M$ such that $G-m$ is critical and $F$ is an espousal of $G-m$.  A path $P$ is $F$-alternating if its edges alternate between $F$ and 
$E\setminus F$
\begin{observation}\label{obs: no infMalter 1obst}
 If $G$ is a $1$-obstruction and $F$ is a matching that fits $G$, then $G$ does not have an infinite $F$-alternating path that starts in 
 $W$ with an edge in $F$.
\end{observation}
\begin{proof}
 Suppose, for a contradiction, that $P$ is such a path. Let $m\in M$ be the unique element of $M\setminus \bigcup F$.
 Then $F':=E(P)\triangle F$ is a matching with $\bigcup F'=M\setminus \{ m \}\cup W\setminus \{ w \}$, where $w$ is the first vertex of $P$. This 
 contradicts the fact that $G-m$ is critical. 
 \end{proof} 

\begin{theorem}[{Podewski and Steffens, \cite[Theorem 7]{podewski1976injective}}]\label{thm: Marriage ctbl}
A countable bipartite graph $G=(M,W,E)$  is inespousable if and only if it has a $1$-obstruction as a saturated subgraph.  
\end{theorem}
Since saturated subgraphs inherit the Milner-Shelah condition,   in order to prove Theorem \ref{thm: MilnerShelah} for countable graphs it 
is sufficient to show the following: 
\begin{claim}\label{clm: 1obstructions violates}
 Every $1$-obstruction violates the Milner-Shelah condition. 
\end{claim}
\begin{proof}
 Suppose, for a contradiction, that $G=(M,W,E)$ is a $1$-obstruction 
that satisfies the Milner-Shelah condition. Let $F$ be a matching that fits $G$ and let $m$ be the unique element of $M \setminus \bigcup F$. 
Define $D=(V',A)$ to be the directed graph obtained  by orienting 
the edges in $E\setminus F$ from $M$ to 
 $W$, deleting the edges in $F$, and then identifying the endpoints of each $e\in F$ to get a single vertex $v_e$. Observe that each vertex of 
 $V'\setminus \{ m \}$ is of the form $v_e$ for some $e\in F$.
We write $N_D^{+}(v)$ and  $N_D^{-}(v)$ for the out-neighbours and in-neighbours of $v$ respectively. The out-degree $deg_{D}^{+}(v)$ and 
in-degree $deg_{D}^{-}(v)$  are the respective sizes of these sets.  By the construction of $D$ and the 
  Milner-Shelah condition for $G$:
  \begin{enumerate}[label=(\arabic*)]
    \item\label{item: inout light}   $deg_{D}^{+}(v)\geq deg_{D}^{-}(v)$ for each $v\in V'\setminus \{ m \}$;
    \item\label{item: source light}  $deg_D^{-}(m)=0$ and $deg_D^{+}(m)>0$ ;
    \item\label{item: inf outdegree light}   If $(u,v)\in A$  and $deg_{D}^{-}(v)$ is infinite, then $deg_{D}^{+}(u) \geq 
    deg_{D}^{-}(v) $.
   \end{enumerate}
  \begin{observation}\label{obs: no inf dirpath ctbl}
   There is no infinite directed path in $D$. 
  \end{observation}
  \begin{proof}
   Such a path would yield an infinite $F$-alternating path in $G$ starting in $W$ with an edge in $F$ contradicting Observation  \ref{obs: no 
   infMalter 1obst}.
  \end{proof}
A directed tour $T$ in $D$ is a finite or infinite sequence $v_0,v_1,v_2,\ldots$ of not necessarily distinct vertices such that the $(v_i, v_{i+1})$ 
are pairwise distinct arcs of $D$. We construct
 an infinite directed tour $T$ in $D$. Let $v_0:=m$ and let $v_1\in N_D^{+}(v_0)$ be arbitrary ($N_D^{+}(v_0)\neq \emptyset$ by \ref{item: 
 source light}).
 Suppose 
 that 
 $v_0,v_1,\dots, v_n$ are 
 already defined for some $n>0$. Then $d^{+}_D(v_n) \geq d^{-}_D(v_n)$  because of \ref{item: inout light}, since $m\neq v_n$ by 
 $deg_{D}^{-}(m)=0$. If there is a $v\in 
 N_D^{+}(v_n)$ such that $v\neq v_i$ for every
 $i\leq n$, then let $v_{n+1}:=v$ for any such $v$. Otherwise $v_{n+1}\in N_D^{+}(v_n)$ is chosen in such a way that
   $(v_n, v_{n+1})\neq (v_i,v_{i+1})$ for $i<n$. Such a choice is always possible by $d^{+}_D(v_n) \geq d^{-}_D(v_n)$. This completes the 
   recursive construction. If in 
 the resulting infinite tour each vertex appears only 
 finitely often, then it has an injective infinite subsequence which provides an infinite directed path in $D$. This contradicts Observation 
 \ref{obs: no inf dirpath ctbl}. Therefore, some vertex $v$ appears infinitely often in $T$.
 Then  $deg_D^{-}(v) $ is infinite because the arcs used by a tour are pairwise distinct. By property \ref{item: inf outdegree light}, $deg_D^{+}(u) 
$ is infinite for every 
  $u\in N_D^{-}(v)$. But then the second visit to $v$ must have been avoided.  Indeed, let $k<n$ 
 with $v=v_k=v_n$. Then $v_{n-1}\in N_D^{-}(v)$ and  
 thus $deg_{D}^{+}(v_{n-1})$ is infinite. But then there exists a $u\in N_D^{+}(v_{n-1})$ with $u\neq v_i$ for $i\leq n$ and $v_{n}$ was
 chosen to be such a $u$ by construction contradicting $v_n=v_k$. 
\end{proof}

\section{Open questions}
The mold of the results in this paper  fits a wide range of problems. As appetizers, let us list two. 

\begin{problem}
Let $\cs=(S_1, \ldots ,S_m)$ be a system of sets, the members of each being unit disks in the plane. Let $q(d)$ be the minimal number $t$ such that 
if $|S_i|\geq t$ for all $i \in [m]$ and $\Delta(\dot{\bigcup}_{j \in [m]} \bigcup S_j)\le d$ then there is an SJR. Find $q(d)$!  How does the answer 
change if we demand that each $S_i$ consists of disjoint disks?
\end{problem}

A simple case:
\begin{problem}
    What is the smallest  $t$ such that if $\cs=(S_i, i \in [m])$ is a family  each whose members is a set of $t$ disjoint unit disks in the plane, and no 
    point belongs to  three disks in $\bigcup_{i \in [m]}S_i$, then $\cs$ has an SJR?
\end{problem}
An obvious challenge is to prove the main results (in particular Theorem \ref{coarseglobal}) not using topology. Another is proving the infinite 
version of Theorem \ref{coarselocal}, namely the hypergraph generalization of the Milner-Shelah theorem.

\newcommand{\Ind}{\operatorname{Ind}}
\newcommand{\DeltaG}{\Delta}

\printbibliography
\end{document}